\documentclass[11pt]{amsart}
\usepackage{amsmath, amssymb}
\usepackage[all]{xy}
\usepackage{verbatim}
\usepackage{multirow}
\usepackage{enumerate}

\newcommand{\R}{\mathbb{R}}
\newcommand{\C}{\mathbb{C}}

\newtheorem{thm}{Theorem}[section]
\newtheorem{lem}[thm]{Lemma}
\newtheorem{prop}[thm]{Proposition}
\newtheorem{defn}[thm]{Definition}

\DeclareMathOperator{\rank}{rank}
\DeclareMathOperator{\range}{range}

\DeclareMathOperator{\tr}{tr}
\DeclareMathOperator{\nullspace}{nullspace}
\newcommand{\tphi}{\tilde{\phi}}
\newcommand{\tpsi}{\tilde{\psi}}
\newcommand{\tsigma}{\tilde{\sigma}}

\newcommand{\nuB}{\nu( \sqrt{
I - \Lambda(1)} B \sqrt{I - \Lambda(1)})=(f,Bf)}
\DeclareMathOperator{\re}{Re}
\author{Christopher Jankowski} \thanks{Supported by the Skirball Foundation via
the Center for Advanced Studies in Mathematics at Ben-Gurion
University of the Negev.}
\address{ \newline Department of Mathematics \newline
Ben-Gurion University of the Negev\newline P.O. Box 653 \newline
Be'er Sheva 84105, Israel} \email{cjankows@math.bgu.ac.il}

\title[Non-cocycle conjugate $E_0$-semigroups]{A family of non-cocycle conjugate 
$E_0$-semigroups obtained from boundary weight doubles}
\begin{document} \maketitle

\begin{abstract}
We have seen that if $\phi: M_n(\C) \rightarrow M_n(\C)$ is a unital
$q$-positive map and $\nu$ is a type II Powers weight, then the boundary weight
double $(\phi, \nu)$ induces a unique (up to conjugacy) type II$_0$
$E_0$-semigroup. Let $\phi: M_n(\C) \rightarrow M_n(\C)$ and $\psi:
M_{n'}(\C) \rightarrow M_{n'}(\C)$ be unital rank one
$q$-positive maps, so for some states $\rho \in M_n(\C)^*$ and $\rho'
\in M_{n'}(\C)^*$, we have $\phi(A)=\rho(A)I_n$ and $\psi(D) = \rho'(D)I_{n'}$
for all $A \in M_n(\C)$ and $D \in M_{n'}(\C)$.  We find that if
$\nu$ and $\eta$ are arbitrary type II Powers weights, then $(\phi, \nu)$
and $(\psi, \eta)$ induce non-cocycle conjugate $E_0$-semigroups
if $\rho$ and $\rho'$ have different
eigenvalue lists.  We then completely classify
the $q$-corners and hyper maximal $q$-corners from $\phi$
to $\psi$, obtaining the following result:  If $\nu$ is a type
II Powers weight of the form
$\nu(\sqrt{I - \Lambda(1)} B \sqrt{I - \Lambda(1)})=(f,Bf)$,  then
the $E_0$-semigroups induced by $(\phi,
\nu)$ and $(\psi, \nu)$ are cocycle conjugate if and only if $n=n'$
and $\phi$ and $\psi$ are conjugate.
\end{abstract}

\section{Introduction}
An $E_0$-semigroup $\alpha=\{\alpha_t\}_{t\geq 0}$ is a
semigroup of unital $*$-endomorphisms of $B(H)$ which is weakly continuous
in $t$.  $E_0$-semigroups are divided into three types, depending on the
existence and structure of their units. More specifically, if
$\alpha$ is an $E_0$-semigroup acting of $B(H)$ and there is a strongly
continuous
semigroup $U=\{U_t\}_{t \geq 0}$ of bounded operators acting on $H$ such that
$\alpha_t(A)U_t=U_tA$ for all $A \in B(H)$ and $t \geq 0$, then we
say that $U$ is a unit for $\alpha$. An $E_0$-semigroup is
said to be spatial if it has at least one unit, and a spatial $E_0$-semigroup
is called completely spatial if, in essence, its
units can reconstruct $H$.  We say an $E_0$-semigroup $\alpha$ is
type I if it is completely spatial and type II if it is spatial but
not completely spatial. If $\alpha$ has no units, we say it is of
type III.  Every spatial $E_0$-semigroup $\alpha$ is assigned an index $n \in
\mathbb{Z}_{\geq 0} \cup \{\infty\}$ which corresponds to the
dimension of a particular Hilbert space associated to its units.
The type I $E_0$-semigroups are classified up to cocycle
conjugacy by their index:  If $\alpha$ is of type I$_n$ (type I,
index $n$) for $n \in \mathbb{N} \cup \{\infty\}$, then $\alpha$ is
cocycle conjugate to the CAR flow of rank $n$ (\cite{arvindex}),
while if $\alpha$ is of type I$_0$, then it is a semigroup of
$*$-automorphisms.

However, uncountably many examples of non-cocycle
conjugate $E_0$-semigroups of types II and III are known (see, for example,
\cite{izumisri}, \cite{izumi}, \cite{bigpaper}, \cite{typeII}, \cite{typeIII},
and \cite{T2}). Bhat's dilation theorem (\cite{Bhat}) and developments in the theory
of $CP$-flows (\cite{hugepaper} and \cite{bigpaper}) have led to the
introduction of boundary weight doubles and related cocycle
conjugacy results for $E_0$-semigroups in \cite{Me}.
\begin{comment}A boundary
weight double is a pair $(\phi, \nu)$, where $\phi: M_n(\C)
\rightarrow M_n(\C)$ is $q$-positive (that is, $\phi(I + t
\phi)^{-1}$ is completely positive for all $t \geq 0$)
and $\nu$ is a positive boundary weight over $L^2(0, \infty)$
(denoted by $\nu \in \mathfrak{A}(L^2(0, \infty))_*^+$).
If, in addition, $\nu$ is normalized and unbounded, we call $\nu$
a type II Powers weight.  If $\phi: M_n(\C) \rightarrow M_n(\C)$
is unital and $q$-positive and $\nu$ is a type II Powers weight, then 
$(\phi, \nu)$ induces a $CP$-flow whose Bhat minimal dilation is
a type II$_0$ $E_0$-semigroup $\alpha^d$.   This means that $\alpha_t ^d$ is a
proper endomorphism for every $t>0$ and that $\alpha^d$ has exactly
one unit up to exponential scaling.
\end{comment}
A boundary
weight double is a pair $(\phi, \nu)$, where $\phi: M_n(\C)
\rightarrow M_n(\C)$ is $q$-positive (that is, $\phi(I + t
\phi)^{-1}$ is completely positive for all $t \geq 0$)
and $\nu$ is a positive boundary weight over $L^2(0, \infty)$.  If $\phi$
is unital and $\nu$ is normalized and unbounded (in which
case we say $\nu$ is a type II Powers weight), then
$(\phi, \nu)$ induces a unital $CP$-flow whose Bhat minimal dilation is
a type II$_0$ $E_0$-semigroup $\alpha^d$.  If $\phi: M_n(\C) \rightarrow M_n(\C)$ is unital and $q$-positive and $U \in M_n(\C)$
is unitary, then the map $\phi_U(A)=U^*\phi(UAU^*)U$ is also unital and $q$-positive.  
The relationship between $\phi$ and $\phi_U$ is analogous
to the definition of conjugacy for $E_0$-semigroups.  With this in mind,
we say that $q$-positive maps $\phi, \psi: M_n(\C)
\rightarrow M_n(\C)$
are \textit{conjugate} if $\psi = \phi_U$ for some unitary $U \in M_n(\C)$.
If $\nu$ is a type II Powers weight of the form $\nuB$, then $(\phi, \nu)$ and
$(\phi_U, \nu)$ induce cocycle conjugate $E_0$-semigroups (for details, 
see Proposition 2.11
of \cite{m2c} and the discussion preceding it).

Suppose $\phi: M_n(\C) \rightarrow M_n(\C)$ and $\psi: M_{n'}(\C)
\rightarrow M_{n'}(\C)$ are unital rank
one $q$-positive maps, so for some states $\rho \in M_n(\C)^*$ and
$\rho' \in M_{n'}(\C)^*$, we have $\phi(A)=\rho(A)I_n$ and
$\psi(D)=\rho'(D)I_{n'}$ for all $A \in M_n(\C), D \in M_{n'}(\C)$.
Let $\nu$ and $\eta$ be type II Powers weights.  We prove three main results. First, we find that if
$(\phi, \nu)$ and $(\psi, \eta)$ induce cocycle conjugate
$E_0$-semigroups, then $\rho$ and $\rho'$ have identical eigenvalue
lists (Definition \ref{eigvallist} and Proposition \ref{kk'}).  We
then find all $q$-corners and hyper maximal $q$-corners from $\phi$
to $\psi$ (see Remark 1 and Theorems \ref{biggie} and \ref{bigone}).  With this
result in hand, we complete the cocycle conjugacy comparison theory for
$E_0$-semigroups $\alpha^d$ and $\beta^d$ induced by $(\phi, \nu)$
and $(\psi, \nu)$ in the case that $\nu$ is of the form $\nu( \sqrt{
I - \Lambda(1)} B \sqrt{I - \Lambda(1)})=(f,Bf)$, finding that
$\alpha^d$ and $\beta^d$ are cocycle conjugate if and only if $n=n'$ and
$\phi$ is conjugate to $\psi$ (Theorem \ref{theone}).

\section{Background}

\subsection{$q$-positive and $q$-pure maps}
Let $\phi: \mathfrak{A} \rightarrow \mathfrak{B}$ be a linear map
between unital $C^*$-algebras.  For each $n \in \mathbb{N}$, define
$\phi_n: M_n(\mathfrak{A}) \rightarrow M_n(\mathfrak{B})$ by

\begin{displaymath} \phi_n \left(\begin{array}{ccc} A_{11} & \cdots & A_{1n}
\\ \vdots & \ddots & \vdots \\ A_{n1} & \cdots & A_{nn}  \end{array}\right)
= \left(\begin{array}{ccc} \phi(A_{11}) & \cdots & \phi(A_{1n})
\\ \vdots & \ddots & \vdots \\ \phi(A_{n1}) & \cdots & \phi(A_{nn})
\end{array} \right).
\end{displaymath}
We say that $\phi$ is completely positive if $\phi_n$ is positive
for all $n \in \mathbb{N}$.  From the work of Choi (\cite{choi}) and Arveson (\cite{arveson}), we know
that every normal completely positive 
map $\phi: B(H) \rightarrow B(K)$ ($H, K$ separable Hilbert spaces) can be
written in the form
$$\phi(A) = \sum_{i=1}^n S_i A S_i^*$$ for
some $n \in \mathbb{N} \cup \{\infty\}$
and bounded operators $S_i: H
\rightarrow K$ which are linearly independent over
$\ell_2(\mathbb{N})$. %in the sense that if $\sum_{i=1}^{r \leq n} z_i S_i = 0$
%for a sequence $\{z_i\}_{i=1}^r \in
%\ell_2(\mathbb{N})$, then $z_i=0$ for all $i$.

We will be interested in a particular kind of completely positive map:
\begin{defn}
Let $\phi: M_n(\C) \rightarrow M_n(\C)$ be a linear map
with no negative eigenvalues.  We say $\phi$ is \emph{$q$-positive} (and write $\phi \geq_q 0$)
if $\phi(I + t \phi)^{-1}$ is completely positive for all $t \geq 0$.
\end{defn}
We make two observations in light of this definition.  First, it is
not uncommon for a completely positive map to have negative
eigenvalues.  Second, there is no ``slowest rate of failure'' for
$q$-positivity:  For every $s \geq 0$, there exists a linear map
$\phi$ with no negative eigenvalues such that $\phi(I + t
\phi)^{-1}$ ($t \geq 0$) is completely positive if and only if $t \leq s$. These
observations are discussed in detail in section 2.1 of \cite{m2c}.  

There is a
natural order structure for $q$-positive maps.  If $\phi, \psi:
M_n(\C) \rightarrow M_n(\C)$ are $q$-positive, we say $\phi$
$q$-dominates $\psi$ (i.e. $\phi \geq_q \psi$) if $\phi(I + t
\phi)^{-1} - \psi(I + t \psi)^{-1}$ is completely positive for all
$t \geq 0$.  It is not always true that $\phi \geq_q \lambda \phi$ if $\lambda \in (0,1)$
(for a large family of counterexamples, see Theorem 6.11 of \cite{Me}).
However, if $\phi$ is $q$-positive, then for every $s
\geq 0$, we have $\phi \geq_q \phi(I + s \phi)^{-1} \geq_q 0$ 
(Proposition 4.1 of \cite{Me}).   If these are the only nonzero
$q$-subordinates of $\phi$, we say $\phi$ is \emph{$q$-pure}.  The unital
$q$-pure maps which are either rank one or invertible have been
classified (Proposition 5.2 and Theorem 6.11 of \cite{Me}).
%We should note that
%there exist $q$-positive maps $\phi$ such that $\phi$ does not $q$-dominate $\lambda \phi$
%for any $\lambda \in (0,1)$ (in particular, this is true for many
%invertible $q$-pure maps).

If $\phi$ is a unital $q$-positive map, then as $t \rightarrow
\infty$, the maps $t \phi(I + t \phi)^{-1}$ converge to an idempotent
completely positive map
$L_\phi$ which has interesting properties (see Lemma 3.1 of \cite{m2c}):

\begin{lem}\label{Lphi}
Suppose $\phi: M_n(\C) \rightarrow M_n(\C)$ is $q$-positive and $||t \phi(I + t \phi)^{-1}||
< 1$ for all $t \geq 0$. Then the maps $t \phi(I + t \phi)^{-1}$
have a unique norm limit $L_\phi$ as $t \rightarrow \infty$, and $L_\phi$
is completely positive.  Furthermore,
\begin{enumerate}[(i)]
\item $\phi=\phi \circ L_\phi = L_\phi \circ \phi$,
\item $L_\phi^2= L_\phi$,
\item $\range(L_\phi)=\range(\phi)$, and
\item $\nullspace(L_\phi) = \nullspace(\phi)$.
\end{enumerate}
\end{lem}

\subsection{$E_0$-semigroups and $CP$-flows}
From a celebrated result of Wigner (\cite{wigner}), we know that every one-parameter
group $\alpha=\{\alpha_t\}_{t\in \R}$ of $*$-automorphisms of $B(H)$ arises
from a strongly continuous unitary group $\{V_t\}_{t \in \R}$ in the sense that
$\alpha_t(A)=V_t A V_t^*$ for all $t \in \R$ and $A \in B(H)$.  

\begin{defn}  Let $H$ be a separable Hilbert space.
We say a family $\alpha = \{\alpha_t\}_{t \geq 0}$ of
$*$-endomorphisms of $B(H)$ is an \emph{$E_0$-semigroup} if:
\begin{enumerate}[(i)]
\item  $\alpha_s \circ \alpha_t = \alpha_{s+t}$ for all $s,t \geq 0$
and $\alpha_0 (A)=A$ for all $A \in B(H)$;
\item  For each $f,g \in H$ and $A \in B(H)$, the inner product
$(f, \alpha_t(A)g)$ is continuous in $t$;
\item $\alpha_t(I) =I$ for all $t \geq 0$.
\end{enumerate}
\end{defn}

We have two notions of equivalence for $E_0$-semigroups:

\begin{defn}
Let $\alpha$ and $\beta$ be $E_0$-semigroups acting on $B(H_1)$ and
$B(H_2)$, respectively, are said to be \emph{conjugate} if there is
a $*$-isomorphism $\theta$ from $B(H_1)$ onto $B(H_2)$ such that
$\theta \circ \alpha \circ \theta^{-1}= \beta$.

We say $\alpha$ and $\beta$ are \emph{cocycle conjugate} if $\alpha$
is conjugate to $\beta'$, where $\beta'$ is an $E_0$-semigroup of
$B(H_2)$ satisfying the following condition:  For some strongly
continuous family of unitaries $W = \{W_t\}_{t \geq 0}$ acting on
$H_2$ and satisfying $W_t \beta_{t}(W_s)=W_{t+s}$, we have
$\beta'_t(A) = W_t \beta_t(A)W_t^*$ for all $A \in B(H_2)$ and
$t\geq 0$.
\begin{comment}Equivalently, $\alpha$ and $\beta$ are cocycle conjugate
if and only if, for some collection $\{\gamma_t\}_{t \geq 0}$ of
maps from $B(H_2, H_1)$ into itself, the family $\Theta =
\{\Theta_t\}_{t \geq 0}$ defined below is an $E_0$-semigroup on
$B(H_1 \oplus H_2)$:
\begin{displaymath}
\Theta_t= \left(\begin{array}{cc} A_{11} & A_{12}
\\ A_{21} & A_{22} \end{array} \right)
\left(\begin{array}{cc} \alpha_t (A_{11}) & \gamma_t(A_{12})
\\ \gamma_t^*(A_{21}) & \beta_t(A_{22}) \end{array} \right).
\end{displaymath}
\end{comment}
\end{defn}

Let $K$ be a separable Hilbert space, and form $H = K \otimes L^2(0,
\infty)$, which we identify with the space of 
$K$-valued measurable functions on $(0, \infty)$ which are square
integrable.  Let $U=\{U_t\}_{t \geq 0}$ be the right shift semigroup
on $H$, so for all $t \geq 0$, $f \in H$, and $x>0$, we have
$$(U_t f)(x) = f(x-t) \textrm{ if } x > t, \ \ (U_t f)(x)=0
\textrm{ if } x \leq t.$$ A strongly continuous semigroup
$\alpha=\{\alpha_t\}_{t \geq 0}$ of completely positive contractions
from $B(H)$ into itself 
is called a \textit{$CP$-flow} over $K$ if $\alpha_t(A)U_t = U_t A$
for all $A \in B(H)$ and $t \geq 0$.  A result of Bhat in \cite{Bhat}
shows that if $\alpha$ is unital, then it minimally dilates to a unique (up to conjugacy)
$E_0$-semigroup $\alpha^d$.  We may naturally construct a $CP$-flow $\beta= \{\beta_t\}_{t \geq
0}$ over $K$ using the right shift semigroup by defining
$$\beta_t(A)= U_t A U_t ^*$$ for all $A \in B(H)$, $t \geq 0$.
In fact, if $\alpha$ is any $CP$-flow over $K$, then $\alpha$
dominates $\beta$ in the sense that $\alpha_t - \beta_t$ is
completely positive for all $t \geq 0$.

Define $\Lambda: B(K) \rightarrow B(H)$ by
$$(\Lambda(A)f)(x)=e^{-x} Af(x)$$ for all $A \in B(K), f \in H,$
and $x \in (0, \infty)$, and let $\mathfrak{A}(H)$ be the algebra
$$\mathfrak{A}(H) = \sqrt{I - \Lambda(I_K)} B(H)  \sqrt{I -
\Lambda(I_K)}.$$  We say a linear functional $\tau$ acting on $\mathfrak{A}(H)$
is a \textit{boundary weight} (denoted $\tau \in \mathfrak{A}(H)_*$) if the functional
$\ell$ defined on $B(H)$ by
$$\ell(A)= \tau\Big(
\sqrt{I - \Lambda(I_K)}A\sqrt{I - \Lambda(I_K)}\Big)$$
satisfies $\ell \in B(H)_*$.  For a discussion of boundary weights and their properties, we refer
the reader to Definition 1.10 of \cite{markie} and the remarks that follow it.

Every $CP$-flow over $K$ corresponds to a \textit{boundary weight
map} $\rho \rightarrow \omega(\rho)$
from $B(K)_*$ to $\mathfrak{A}(H)_*$ (\cite{hugepaper}). On the other hand, it is an extremely important
and non-trivial fact that, under
certain conditions, a map from $B(K)_*$ to $\mathfrak{A}(H)_*$ can
induce a $CP$-flow (see Theorem 3.3 of \cite{bigpaper}):
\begin{thm}\label{powersthm}
If $\rho \rightarrow \omega(\rho)$ is a completely positive mapping
from $B(K)_*$ into $\mathfrak{A}(H)_*$ satisfying
$\omega(\rho)(I-\Lambda(I_K)) \leq \rho(I_K)$ for all positive
$\rho$, and if the maps $$\hat{\pi}_t : = \omega_t(I +
\hat{\Lambda}\omega_t)^{-1}$$ are completely positive contractions
from $B(K)_*$ into $B(H)_*$ for all $t >0$, then $\rho \rightarrow
\omega(\rho)$ is the boundary weight map of a $CP$-flow over $K$.
The $CP$-flow is unital if and only if $\omega(\rho)(I-\Lambda(I_K))
= \rho(I_K)$ for all $\rho \in B(K)_*$.
\end{thm}

If $\alpha$ is a $CP$-flow over $\C$, then we identify its boundary
weight map $c \rightarrow \omega(c)$ with the single positive
boundary weight $\omega:=\omega(1)$, so $\omega$ has the form $$\omega(\sqrt{I - \Lambda(1)}A
\sqrt{I - \Lambda(1)}) = \sum_{i=1}^k (f_i, Af_i)$$ for some mutually
orthogonal nonzero $L^2$-functions $\{f_i\}_{i=1}^k$ ($k \in \mathbb{N} \cup
\{\infty\}$) with $\sum_{i=1}^k ||f_i||^2<\infty$.   We call $\omega$ a
\textit{positive boundary weight over $L^2(0, \infty)$}, and, following
the notation of \cite{markie}, we write $\omega \in \mathfrak{A}(L^2(0,
\infty))_*^+$.  We say $\omega$ is bounded if there exists some $r>0$
such that $|\omega(B)|\leq r ||B||$ for all $B \in \mathfrak{A}(H)$.
Otherwise, we say $\omega$ is unbounded.
 Suppose $\omega(I - \Lambda(1))=1$ (i.e. $\omega$ is \textit{normalized}), so $\alpha$ is unital and
therefore dilates to an $E_0$-semigroup $\alpha^d$. Results from
\cite{hugepaper} show that $\alpha^d$ is of type I$_k$ if $\omega$
is bounded but of type II$_0$ if $\omega$ is unbounded, leading us to make the following
definition:

\begin{defn} A boundary weight $\nu \in \mathfrak{A}(L^2(0, \infty))_*$ is called
a \emph{Powers
weight} if $\nu$ is positive and normalized.  We say a Powers weight $\nu$ is
\emph{type I}
if it is bounded and \emph{type II} if it is unbounded.
\end{defn}

We note that if $\nu$ is a
type II Powers weight, then both $\nu_t(I)$ and $\nu_t(\Lambda(1))$
approach infinity as $t \rightarrow 0+$.  We can combine unital $q$-positive
maps with type II Powers weights to obtain $E_0$-semigroups (see Proposition
3.2 and Corollary 3.3 of \cite{Me}):

\begin{prop} \label{bdryweight} Let $H =\C^n \otimes L^2(0, \infty)$.
Let $\phi: M_n(\C) \rightarrow M_n(\C)$ be a unital $q$-positive map, and let $\nu$ be a type
II Powers weight.  Let $\Omega_\nu: \mathfrak{A}(H) \rightarrow M_n(\C)$
be the map that sends $A=(A_{ij}) \in M_n(\mathfrak{A}(L^2(0, \infty))) \cong \mathfrak{A}(H)$
to the matrix $(\nu(A_{ij})) \in M_n(\C)$.  The map $\rho \rightarrow
\omega(\rho)$ from $M_n(\C)^*$ into
$\mathfrak{A}(H)_*$ defined by
$$\omega (\rho) (A) = \rho\Big(\phi(\Omega_\nu(A))\Big)$$
is the boundary weight map of a unital $CP$-flow $\alpha$ over $\C^n$ 
whose Bhat minimal dilation $\alpha^d$ is a type
II$_0$ $E_0$-semigroup.
\end{prop}

In the notation of the previous proposition, we say that $\alpha^d$ is the
$E_0$-semigroup induced by the \textit{boundary weight double} $(\phi, \nu)$.

\begin{defn}  Suppose $\phi: B(H_1) \rightarrow B(K_1)$ and $\psi: B(H_2) \rightarrow
B(K_2)$ are normal completely positive maps.  Write each $A \in B(H_1 \oplus H_2)
$ as $A=(A_{ij})$, where $A_{ij} \in B(H_j, H_i)$ for each
$i,j=1,2$.  We say a linear map $\gamma:
B(H_2, H_1) \rightarrow B(K_2, K_1)$ is a \emph{corner} from $\alpha$ to
$\beta$ if $\Theta: B(H_1 \oplus H_2) \rightarrow B(K_1 \oplus K_2)$
defined by
\begin{displaymath} \Theta \left( \begin{array}{cc} A_{11} & A_{12} \\
A_{21} & A_{22} \end{array}  \right) = \left(
\begin{array}{cc} \phi(A_{11}) & \gamma(A_{12}) \\ \gamma^*(A_{21})
& \psi(A_{22}) \end{array} \right)
\end{displaymath} is normal and completely positive.

Suppose $H_1=K_1= \C^n$ and $H_2=K_2=\C^{m}$.  We say
$\gamma: M_{n,m}(\C) \rightarrow M_{n,m}(\C)$ is
a \emph{$q$-corner} from $\phi$ to $\psi$ if $\Theta \geq_q 0$.  A
$q$-corner $\gamma$ is \emph{hyper maximal} if, whenever
\begin{displaymath}
\Theta \geq_q \Theta' = \left( \begin{array}{cc} \phi' & \gamma
\\ \gamma^* & \psi' \end{array} \right) \geq_q 0, \end{displaymath}
we have $\Theta= \Theta'$.
\end{defn}

Hyper maximal $q$-corners between unital $q$-positive maps $\phi$ and $\psi$
allow us to compare $E_0$-semigroups induced
by $(\phi, \nu)$ and $(\psi, \nu)$ if $\nu$ is a particular kind
of type II Powers weight:
\begin{prop}\label{hypqc}
Let $\phi: M_n(\C) \rightarrow M_n(\C)$ and $\psi: M_k(\C)
\rightarrow M_k(\C)$ be unital $q$-positive maps, and let $\nu$ be a
type II Powers weight of the form $$\nu(\sqrt{I -
\Lambda(1)}B\sqrt{I - \Lambda(1)})=(f, Bf).$$ The boundary weight
doubles $(\phi, \nu)$ and $(\psi, \nu)$ induce cocycle conjugate
$E_0$-semigroups if and only if there is a hyper maximal $q$-corner
from $\phi$ to $\psi$.
\end{prop}

From \cite{Me}, we know that a unital rank one map $\phi: M_n(\C)
\rightarrow M_n(\C)$ is $q$-positive if and only if it has the form
$\phi(A)=\rho(A)I$ for a state $\rho \in M_n(\C)^*$, and that $\phi$
is $q$-pure if and only if $\rho$ is \textit{faithful}.  We also
have the following comparison result (Theorem 5.4 of \cite{Me}),
which we will extend in this paper to all unital rank one
$q$-positive maps (Theorem \ref{theone}):
\begin{thm}\label{statesbig}  Let $\phi: M_n(\C) \rightarrow M_n(\C)$ and
$\psi: M_{n'}(\C) \rightarrow M_{n'}(\C)$ be rank one unital $q$-pure
maps, so for some faithful
states $\rho \in M_n(\C)^*$ and $\rho' \in M_{n'}(\C)^*$, we have
$$\phi(A)=\rho(A)I_n \ \textrm{ and } \ \psi(D) = \rho'(D)I_{n'}$$ for
all $A \in M_n(\C)$ and $D \in M_{n'}(\C)$ .  Let $\nu$
be a type II Powers weight of the form $\nu(\sqrt{I - \Lambda(1)} B
\sqrt{I - \Lambda(1)}) = (f,Bf)$.

The boundary weight doubles $(\phi, \nu)$ and $(\psi, \nu)$
induce cocycle conjugate $E_0$-semigroups if and only if $n=n'$ and for some
unitary $U \in M_n(\C)$ we have $\rho'(A)=\rho(UAU^*)$ for all $A \in M_n(\C)$.
\end{thm}

\subsection{Conjugacy for $q$-positive maps}

We will only be concerned with the identity of a $q$-positive map up to
an equivalence relation we will call conjugacy.  More specifically,
if $\phi: M_n(\C) \rightarrow M_n(\C)$ is a unital $q$-positive map
and $U \in M_n(\C)$ is any unitary matrix, the map $\phi_U(A) := U^*
\phi(UAU^*)U$ is also unital and $q$-positive.   We have the following definition from \cite{m2c}:

\begin{defn}
Let $\phi, \psi: M_n(\C) \rightarrow M_n(\C)$ be $q$-positive maps.  We say
$\phi$ is \emph{conjugate} to $\psi$ if $\psi=\phi_U$ for some unitary $U \in M_n(\C)$.
\end{defn}

Conjugacy is clearly an equivalence relation, and its definition is
analogous to that of conjugacy for $E_0$-semigroups. Indeed, since
every $*$-isomorphism of $M_n(\C)$ is implemented by unitary
conjugation, two $q$-positive maps $\phi, \psi: M_n(\C) \rightarrow
M_n(\C)$ are conjugate if and only if $\psi= \theta \circ \phi \circ \theta^{-1}$
for some $*$-isomorphism
$\theta$ of $M_n(\C)$.
If $\nu$ is a type II Powers weight
of the form $\nu(\sqrt{I - \Lambda(1)}B\sqrt{I -
\Lambda(1)})=(f,Bf)$, then conjugacy between unital $q$-positive maps
$\phi$ and $\psi$ is always a sufficient condition for
$(\phi, \nu)$ and $(\psi, \nu)$ to induce cocycle conjugate $E_0$-semigroups.
Indeed, it is
straightforward to verify that if $\phi: M_n(\C) \rightarrow M_n(\C)$
is unital and $q$-positive, then the map $\gamma: M_n(\C) \rightarrow
M_n(\C)$ defined by $\gamma(A)=\phi(AU^*)U$ is a hyper maximal
$q$-corner from $\phi$ to $\phi_U$ (for details, see the discussion preceding
Proposition 2.11 of
\cite{m2c}), whereby Proposition \ref{hypqc} gives us:

\begin{prop}\label{arrgh} Let $\phi: M_n(\C) \rightarrow M_n(\C)$
be unital and $q$-positive, and suppose $\psi$
is conjugate to $\phi$.  If $\nu$ is a type II Powers weight
of the form $$\nu(\sqrt{I - \Lambda(1)}B\sqrt{I -
\Lambda(1)})=(f,Bf),$$ then $(\phi, \nu)$ and $(\psi, \nu)$ induce
cocycle conjugate $E_0$-semigroups.
\end{prop}

In the case that $\phi$ and $\psi$ are unital rank one $q$-pure maps
and
$\nu$ is a type II Powers weight
of the form $\nu(\sqrt{I - \Lambda(1)}B\sqrt{I -
\Lambda(1)})=(f,Bf)$,
Theorem \ref{statesbig} states that conjugacy between $\phi$ and $\psi$ is
both necessary and sufficient
for $(\phi, \nu)$ and $(\psi, \nu)$ induce
cocycle conjugate $E_0$-semigroups.

Let
$\phi: M_n(\C) \rightarrow M_n(\C)$ be a unital linear map of rank one. 
It is not difficult to see that $\phi$ is $q$-positive if and only if it
has the form $\phi(A)=\rho(A)I$ for some state $\rho \in M_n(\C)^*$. It is
well-known that we can write $\rho$ in the form \begin{equation}
\label{rhoform} \rho(A)= \sum_{i=1}^{k \leq n} \lambda_i (g_i, A
g_i),\end{equation} for some mutually orthogonal unit vectors
$\{g_i\}_{i=1}^{k} \subset \C^n$ and some positive numbers
$\lambda_1 \geq \cdots \geq \lambda_k > 0$ such that $\sum_{i=1}^{k
} \lambda_i = 1$. With the conditions of the previous sentence
satisfied, the number $k$ and the monotonically decreasing set
$\{\lambda_i\}_{i=1}^k$ are unique.

\begin{defn} \label{eigvallist}  Assume the notation of the previous paragraph.  We call
$\{\lambda_i\}_{i=1}^k$ the \emph{eigenvalue
list} for $\rho$.
\end{defn}

We should note that our definition differs from a previous definition of eigenvalue list 
in the literature (see, for example, \cite{morearv})
in that our eigenvalue lists do not include zeros.  By our definition,
is possible for states $\rho$ and $\rho'$ acting on $M_n(\C)$ and $M_{n'}(\C)$
to have identical eigenvalue lists if $n \neq n'$.

Let $\{e_i\}_{i=1}^n$ be the standard basis for $\C^n$.
If $\rho$ has the form \eqref{rhoform} and $U \in M_n(\C)$ is any unitary matrix
such that $Ue_i= g_i$ for all $i=1, \ldots, k$, then 
$$\rho(UAU^*)=\sum_{i=1}^k \lambda_i (g_i, UAU^*g_i)=
\sum_{i=1}^k \lambda_i (U^* g_i, AU^*g_i) = \sum_{i=1}^k \lambda_i
(e_i, A e_i)$$ and
\begin{equation}\label{di} \phi_U(A) = U^* \phi(UAU^*)U = U^*\Big[ \Big( \sum_{i=1}^k \lambda_i (e_i, A e_i) \Big)
 I \Big] U
= \Big(\sum_{i=1}^k \lambda_i a_{ii} \Big) I\end{equation} for all $A \in M_n(\C)$.
We will use this fact repeatedly.

%Therefore,
%when working with unital rank one $q$-positive maps, we may assume
%that their implementing states are diagonal, and of course that
%their eigenvalue list is monotonically decreasing.

\begin{comment}In the proof of Theorem 5.4 of \cite{Me}, it was shown that if $\phi: M_n(\C)
\rightarrow M_n(\C)$ and $\psi: M_{n'}(\C) \rightarrow M_{n'}(\C)$
are implemented by faithful states $\rho$ and $\tau$, then there is
a norm one corner $\gamma$ from $\phi$ to $\psi$ if and only if
$n=n'$ and $\tau = \rho_U$ for some unitary $U \in M_n(\C)$, which
is to say that the eigenvalues of $\rho$ and $\tau$ are identical up
to some permutation.  Since we are assuming that eigenvalue lists
are monotonically decreasing sequences, the previous sentence gives
us:
\end{comment}

\section{Our results}
We begin with the following observation:
\begin{lem} \label{normone} Let $\phi: M_n(\C) \rightarrow M_n(\C)$ and $\psi: M_{n'}(\C)
\rightarrow M_{n'}(\C)$ be unital $q$-positive maps, and let $\nu$
and $\eta$ be type II Powers weights.  If the boundary weight
doubles $(\phi, \nu)$ and $(\psi, \eta)$ induce cocycle conjugate
$E_0$-semigroups, there is a corner $\gamma$ from $L_\phi$ to
$L_\psi$ such that $||\gamma||=1$ and $1$ is an eigenvalue of $\gamma$.
\end{lem}

\begin{proof}  This is a slight generalization of Lemma
5.3 of \cite{Me} (where $\phi$ and $\psi$ were assumed to have rank one
and be $q$-pure), but its proof is identical.
%In Lemma 5.3, $\phi$ and $\psi$ were assumed to have
%rank one, in which case we know for all $t>0$ that $\nu_t(\Lambda(1)) \phi(I + \nu_t(\Lambda(1)) \phi)^{-1}
%= \frac{\nu_t(\Lambda(1))}{1+ \nu_t(\Lambda(1))} \phi$ and
%$L_\phi = \phi$ (similarly, $L_\psi=\psi$).
%However, if we merely assume $\phi$ and $\psi$ are any unital $q$-positive
%maps,
Indeed, the exact same argument as in the proof of Lemma 5.3 shows
that there is a corner $\gamma$ from $\lim_{t
\rightarrow 0^+} \nu_t(\Lambda(1)) \phi(I + \nu_t(\Lambda(1))
\phi)^{-1}$ to $\lim_{t \rightarrow 0^+} \eta_t(\Lambda(1)) \psi(I +
\eta_t(\Lambda(1)) \psi)^{-1}$ (provided the limits exist) such that $||\gamma||=1$
and 1 is an eigenvalue of $\gamma$.  
We observe that the former limit is $L_\phi$ and
the latter limit is $L_\psi$.  Indeed, the values $\{\nu_t(\Lambda(1))\}_{t >0}$
and $\{\eta_t(\Lambda(1))\}_{t >0}$ are monotonically decreasing in $t$, and 
since $\nu$ and $\eta$ are unbounded, we have $$\lim_{t \rightarrow 0+} \nu_t(\Lambda(1))= \lim_{t
\rightarrow 0+} \eta_t(\Lambda(1)) = \infty.$$ \end{proof}

We have the following lemma (see Lemma 3.5 of \cite{Me}):
\begin{lem}\label{corners}
Let $\phi: M_n(\C) \rightarrow M_{r}(\C)$, $\psi: M_{n'}(\C)
\rightarrow M_{r'}(\C)$ be completely positive maps, so for some
$k, k' \in \mathbb{N}$ and sets of linearly independent matrices
$\{S_i\}_{i=1}^{k} \subset M_{r,n}(\C)$ and $\{T_i\}_{i=1}^{k'}
\subset M_{r',n'}(\C)$, we have
\begin{equation}
\label{cornereq} \phi(A) = \sum_{i=1}^k S_i A S_i^*, \ \ \psi(D) = \sum_{i=1}^{k'}
T_i A T_i^* \end{equation}
for all $A \in M_n(\C)$, $D \in M_{n'}(\C).$

A linear map $\gamma: M_{n,n'}(\C) \rightarrow M_{r,r'}(\C)$
is a corner from $\phi$ to $\psi$ if and only if, for some $C=(c_{ij}) \in M_{k,k'}(\C)$
with $||C|| \leq 1$,
we have
$$\gamma(B) = \sum_{i=1}^k \sum_{j=1}^{k'} c_{ij} S_i B T_j^*$$
for all $B \in M_{n,n'}(\C)$.
\end{lem}

 %The existence of a $q$-corners (and hyper maximal $q$-corners)
%between $\phi$ and $\psi$ is a property invariant under conjugacy.  More precisely,
{\bf Remark 1:}  Suppose $\gamma$ is a $q$-corner from $\phi$ to $\psi$.  Let $U \in M_n(\C)$
and $V \in M_{n'}(\C)$ be arbitrary unitary matrices, and let
\begin{displaymath}
\vartheta = \left( \begin{array}{cc} \phi & \gamma \\ \gamma^* & \psi  \end{array} \right) \geq_q 0.
\end{displaymath}
For the unitary matrix
\begin{displaymath}
Z = \left( \begin{array}{cc} U & 0_{n,n'} \\ 0_{n',n} & V \end{array} \right) \in M_{n+n'}(\C),
\end{displaymath}
we have $\vartheta_Z \geq_q 0$ (since $\vartheta \geq_q 0$), where
\begin{displaymath}
\vartheta_Z
\left( \begin{array}{cc} A & B \\ C & D   \end{array} \right)
= \left( \begin{array}{cc} \phi_U(A) & U^* \gamma(UBV^*)V \\ V^* \gamma^*(VCU^*)U & \psi_V(D)  \end{array} \right).
\end{displaymath}
Therefore, $B \rightarrow U^* \gamma(UBV^*)V$ is a $q$-corner from $\phi_U$ to $\psi_V$.
By Proposition 4.5 of \cite{Me}, there is an isomorphism between the $q$-subordinates of $\vartheta$
and the $q$-subordinates of $\vartheta_Z$.  In particular, if $\Phi: M_{n+n'}(\C) \rightarrow M_{n+n'}(\C)$
is a linear map, then $\vartheta \geq_q \Phi \geq_q 0$ if and only if $\vartheta_Z \geq_q \Phi_Z \geq_q 0$. 
It follows that $\gamma$ is a hyper 
maximal $q$-corner from $\phi$ to $\psi$
if and only if $B \rightarrow U^*\gamma(UBV^*)V$ is a 
hyper maximal $q$-corner from $\phi_U$ to $\psi_V$.
The same argument just used also gives us a bijection between norm one corners from $\phi$ to $\psi$
and norm one corners from $\phi_U$ to $\psi_V$.

\begin{prop} \label{kk'}
Let $\phi: M_n(\C) \rightarrow M_n(\C)$
and $\psi: M_{n'}(\C) \rightarrow M_{n'}(\C)$ be unital rank one
$q$-positive maps,
so for some states $\ell \in M_n(\C)^*$ and $\ell' \in M_{n'}(\C)^*$
with eigenvalue lists $\{\lambda_i\}_{i=1}^k$
and $\{\mu_i\}_{i=1}^{k'}$, respectively,
we have
$$\phi(A) = \ell(A)I_n,  \ \ \psi(D)= \ell'(D) I_{n'}$$
for all $A \in M_n(\C)$ and $D \in M_{n'}(\C)$.  Let $\nu$ and $\eta$ 
be type II Powers weights. 

If the boundary
weight doubles $(\phi, \nu)$ and $(\psi, \eta)$ induce cocycle
conjugate $E_0$-semigroups $\alpha^d$ and $\beta^d$, then $k=k'$ and $\lambda_i = \mu_i$ for
all $i=1, \ldots, k$.
\end{prop}

\begin{proof}  Our proof is similar to the proof of 
Theorem 5.4 of \cite{Me}.  Suppose $\alpha^d$ and $\beta^d$ are cocycle
conjugate.  For some unitaries $U \in M_n(\C)$ and $V \in M_{n'}(\C)$, we have
$$\phi_U(A)= \Big(\sum_{i=1}^k \lambda_i a_{ii}\Big)I_n
, \ \ \psi_V(D) =
\Big(\sum_{i=1}^{k'} \mu_i b_{ii}\Big)I_{n'}$$ for all $A \in M_n(\C)$ and $D \in M_{n'}(\C)$.
Let $\{e_i\}_{i=1}^n$ and $\{e_i'\}_{i=1}^{n'}$ be the standard bases
for $\C^n$ and $\C^{n'}$, respectively, and let $\rho \in M_n(\C)^*$ and 
$\rho' \in M_{n'}(\C)^*$ be the functionals \begin{equation} \label{jj} \rho(A)  
= \sum_{i=1}^k \lambda_i e_i ^* A e_i = \sum_{i=1}^k \lambda_i a_{ii}, \ \
\rho'(D) = \sum_{i=1}^{k'} \mu_i e_i'^*De_i'=\sum_{i=1}^{k'} \mu_i d_{ii},
\end{equation} 
so $\phi_U(A) = \rho(A)I_n$ and $\psi_V(D)=\rho'(D)I_{n'}$
for all $A \in M_n(\C)$ and $D \in M_{n'}(\C)$.
Note that $L_\phi = \phi$ and $L_\psi = \psi$, so 
by Lemma \ref{normone}, there is a norm one corner from $\phi$ to $\psi$.  Therefore, 
by
Remark 1, there is a norm one corner $\gamma$ from $\phi_U$
to $\psi_V$, so the map $\Theta: M_{n+n'}(\C) \rightarrow M_{n+n'}(\C)$
defined by
\begin{equation*} \Theta \left( \begin{array}{cc} A_{n,n} & B_{n,n'} \\ 
C_{n',n} & D_{n',n'} \end{array} \right)  = 
\left( \begin{array}{cc} \rho(A)I_n &  \gamma(B) \\ 
\gamma^*(C) & \rho'(D)I_{n'} \end{array} \right) 
\end{equation*}
is completely positive.  

Since $||\gamma||=1$, there is some $X \in M_{n , n'}(\C)$
with $||X||=1$ and some unit vector $g \in \C^{n'}$ such that
$||\gamma(X)g||^2= (\gamma(X)g, \gamma(X)g)=1$.  Let $\tau \in
M_{n,n'}(\C)^*$ be the functional defined by
$$\tau(B) = (\gamma(X)g, \gamma(B)g).$$
%Defining $S: \C^2 \rightarrow \C^{n+n'}$ by
%\begin{displaymath}
% S \left( \begin{array}{c} c_1 \\ c_2 \end{array} \right)= \left(
%\begin{array}{c} c_1 \gamma(X)g\\ c_2 g \end{array} \right),
%\end{displaymath}
Letting 
\begin{equation*}
S= \left(\begin{array}{cc} \gamma(X)g & 0_{n,1} \\ 0_{n',1}
& g \end{array} \right) \in M_{n+n',2}(\C),
\end{equation*}
we observe that
\begin{displaymath}
\left( \begin{array}{cc} \rho(A) & \tau(B) \\ \tau^*(C) & \rho'(D)
\end{array} \right) = S^* \Theta \left( \begin{array}{cc} A & B \\ C & D
\end{array} \right) S  \ \ \textrm{for all} \
 \left( \begin{array}{cc} A & B \\ C &
D
\end{array} \right) \in M_{n+n'}(\C),
\end{displaymath}
hence $\tau$ is a corner from $\rho$ to $\rho'$.  Note that
$||\tau|| = \tau(X)=1$.

Let $D_\lambda \in M_k(\C)$ and $D_\mu \in M_{k'}(\C)$ be the
diagonal matrices whose $ii$ entries are $\sqrt{\lambda_i}$ and
$\sqrt{\mu_i}$, respectively.  %Defining $\{S_i\}_{i=1}^k \subset M_{1 \times n}(\C)$
%and $\{T_j\}_{j=1}^{k'} \subset M_{1 \times n'}(\C)$ by
%$S_i = \sqrt{\lambda_i}e_i ^*$ and $T_j = \sqrt{\mu_j}(e_j')^*$, 
%we have $$\rho(A)= \sum_{i=1}^k S_i A S_i^*, \ \ \rho'(D) = \sum_{j=1}^k
%T_j D T_j^*$$ for all $A \in M_n(\C)$ and $D \in M_{n'}(\C)$.  
Since $\tau$ is a corner from $\rho$ to $\rho'$, equation \eqref{jj} and Lemma 
\ref{corners} imply that $\tau$
has the form $\tau(B)= \sum_{i,j} c_{ij} \sqrt{\lambda_i \mu_j}
(e_i, B e_j')$ for some $C=(c_{ij}) \in M_{k, k'}(\C)$ such that $||C|| \leq 1$. 
For each $B \in M_{n , n'}(\C)$,
let $\tilde{B} \in M_{k',k}(\C)$ be the top left $k' \times k$
minor of $B^T$, observing that
$$\tau(B)= \sum_{i=1}^k \sum_{j=1}^{k'} c_{ij} \sqrt{\lambda_i \mu_j} b_{ij}
= \tr(C D_\mu \tilde{B} D_\lambda)= \tr\Big(C D_\mu (D_\lambda
(\tilde{B})^*)^* \Big).$$ Let $M=\tilde{X} \in M_{k',k}(\C)$.  Applying the
Cauchy-Schwarz inequality to the inner product $\langle A, B
\rangle=\tr(BA^*)$ on $M_{k,k'} (\C)$, we see
\begin{eqnarray} 1& = & |\tau(X)|^2 = |\tr(C D_\mu (D_\lambda
M^*)^*) |^2 =|\langle D_\lambda M^*, C D_\mu \rangle|^2 \nonumber \\
& \leq & ||CD_\mu||^2_{\tr} ||D_\lambda
M^*||^2_{\tr}= \tr(D_\mu C^* C D_\mu)
\tr(D_\lambda
M^*M D_\lambda) \nonumber \\
\label{ii} & \leq & \tr(D_{\mu} I_{k'} D_{\mu})\tr(D_{\lambda}I_k D_\lambda)
= \Big(\sum_{i=1}^{k'} \mu_i\Big)\Big(\sum_{i=1}^k \lambda_i\Big) = 1 * 1 = 1.
\end{eqnarray}
Since equality holds in Cauchy-Schwarz, it follows that
for some $m \in \C$,
\begin{equation}\label{rk} m C D_\mu = D_\lambda M^*,\end{equation}
where $|m|=1$ since $||CD_\mu||_{\tr} = ||D_\lambda M^*||_{\tr}=1$. In
fact, $m=1$ since $\tau(X)=1$.

Since equality holds in \eqref{ii} and the trace map is faithful, we
have $C^*C= I_{k'}$ and $M^*M = I_{k}$.  Note that $$\min\{k,k'\} \geq
\rank(C)=k', \ \ \min\{k,k'\} \geq \rank(M)=k,$$ hence $k=k'$ and the
previous sentence shows that $C$ and $M$ are unitary.  Therefore,
from \eqref{rk} we have
$$D_\mu = C^* D_\lambda M^* = C^* M^* (M D_\lambda M^*),$$
whereby uniqueness of the right polar decomposition for the invertible positive matrix $D_\mu$
implies $D_\mu = M D_\lambda M^*$.  Since the eigenvalues of $D_\mu$ and $D_\lambda$
are listed in decreasing order, we have $D_\mu = D_\lambda$, hence
$\lambda_i = \mu_i$ for all $i=1, \ldots, k$.
\end{proof}

{\bf Remark 2:} If $\phi: M_k(\C) \rightarrow M_k(\C)$ is a unital rank one $q$-pure map, and if $\gamma$ is a nonzero $q$-corner from $\phi$
to $\phi$, then by Lemma \ref{Lphi}, $\sigma := \lim_{t \rightarrow \infty} t \gamma(I + t \gamma)^{-1}$
is a corner from $\phi$ to $\phi$ satisfying $\sigma^2=\sigma$.  We note that $||\sigma||=1$.
Indeed, since $\sigma^2=\sigma$ and $\range(\sigma) = \range(\gamma) \supsetneq \{0\}$, 
we have $||\sigma|| \geq 1$, while the fact that $\sigma$ is a corner between norm one completely
positive maps
implies $||\sigma|| \leq 1$, hence $||\sigma||=1$.  The following
lemma gives us the form of $\sigma$:
\begin{lem}\label{corner}
Let $\phi: M_k(\C) \rightarrow M_k(\C)$ be a unital $q$-positive
map of the form $\phi(A)=\rho(A)I$.  Assume $\rho$ is a faithful state
of the form
$$\rho(A) = \sum_{i=1}^k \mu_i a_{ii},$$
where $\mu_1, \ldots, \mu_k$ are positive numbers and
$\sum_{i=1}^k \mu_i = 1$.  Let $D_\mu$ be the diagonal matrix with
$ii$ entries $\sqrt{\mu_i}$ for $i=1,\ldots,k$, so $\Omega:= (D_\mu)^2$
is the trace density matrix for $\rho$.

Let $\sigma: M_k(\C) \rightarrow M_k(\C)$ be
a nonzero linear map such that $\sigma^2=\sigma$.
Then $\sigma$ is a corner from $\phi$ to $\phi$ if, and only if,
some unitary $X \in M_k(\C)$ that commutes with $\Omega$, we have
$$\sigma(B)=\tr(X^* B \Omega)X$$ for all $B \in M_k(\C)$.
\end{lem}
\begin{proof}  For the forward direction, suppose
that $\sigma$ is a nonzero corner from $\phi$ to $\phi$ and $\sigma^2=\sigma$.
Note that $||\sigma||=1$ by Remark 2.
We first show that $\sigma$ has rank one. If $\rank(\sigma) \geq
2$, then there is a non-invertible element $A \in \range(\sigma)$.  Scaling $A$
if necessary, we may assume $||A||=1$.  Let $P$ be the orthogonal projection onto the range of
$A$, so $PA=A$ and $A^*=A^*P$.  Since $P \neq I$ and $\rho$ is
faithful, we have $\phi(P) = \rho(A) I = a I$ for some $a<1$. We
note that
\begin{displaymath} 
\left( \begin{array}{cc} P & 0 \\ 0 & I \end{array} \right) \left(
\begin{array}{cc} I & A \\ A^* & I \end{array} \right) \left(
\begin{array}{cc} P & 0 \\ 0 & I \end{array} \right) = \left(
\begin{array}{cc} P & PA \\ A^*P & I \end{array} \right) = \left(
\begin{array}{cc} P & A \\ A^* & I \end{array} \right) \geq 0,
\end{displaymath}
so by complete positivity of $\Theta$ and the fact that
$\sigma^2=\sigma$, we have
\begin{displaymath}
\left( \begin{array}{cc} \phi(P) & \sigma(A) \\ \sigma^*(A^*) &
\phi(I) \end{array} \right)
 = \left( \begin{array}{cc} a I & A \\ A^* & I \end{array} \right)
\geq 0,
\end{displaymath}
which is impossible since $a <1$ and $||A||=1$.  This shows that not
only does $\sigma$ have rank one, but that every non-zero element of its
range is invertible.  In other words, for some linear functional $\tau \in M_k(\C)^*$
and some invertible matrix $X \in M_{k}(\C)$ with $||X||=1$, we have
$\sigma(B)= \tau(B) X$ for all $B \in M_{k}(\C)$.  Since $\sigma$ fixes its range
and $||\sigma||=1$, we have $||\tau||=\tau(X)=1$.

Let $g \in \C^k$ be a unit vector such that $||Xg||=1$. We observe
that $\tau$ is merely the functional $\tau(B) = (\sigma(X)g,
\sigma(B)g)$ for all $B \in M_k(\C)$, and an argument analogous
to the one given in the proof of Proposition \ref{kk'} shows that
$\tau$ is a corner
$\rho$ to $\rho$. By Lemma \ref{corners}, there is some $C \in
M_k(\C)$ with $||C|| \leq 1$ such that
\begin{eqnarray*} \tau(B)  =  \sum_{i,j=1}^k c_{ij} \sqrt{\mu_i \mu_j}
(e_i, B e_j) = \tr(C D_\mu B^T D_\mu)
\end{eqnarray*}
for all $A \in M_k(\C)$.  By the above equation and the fact that $\tau(X)=1$, we may
use the exact same Cauchy-Schwarz argument as in the proof of Proposition \ref{kk'} to conclude
that $C$ and $X^T$ are unitary and that
$$D_\mu= C^* D_\mu (X^T)^* = C^*(X^T)^* (X^T D_\mu (X^T)^*).$$
\begin{comment}: Let $\langle \ , \ \rangle$ be the trace inner
product on $M_k(\C)$. Noting that
$$1=\tau(X)=\tr(C^T D_\mu X D_\mu)= \langle D_\mu X^*, C^T D_\mu
\rangle \leq \tr(D_\mu^2) \tr(D_\mu^2) =1,$$ we see from
Cauchy-Schwartz that $m C^T D_\mu = D_\mu X^*$ for some $m \in \C$
(where $m=1$ since $\tau(X)=1$), and from faithfulness of the trace
we conclude that $C^T$ and $X^*$ are unitary.  This gives us
\begin{equation} \label{indy} D_\mu = (C^T)^* D_\mu X^* = (C^T)^*X^*(X D_\mu X^*). \end{equation}
for some $m \in \C$ with $|m|=1$.  In fact, $m=1$ since $\tau(X)=1$.
\end{comment}
Uniqueness of the polar decomposition for the invertible positive
matrix $D_\mu$ gives us $C^*(X^T)^*= I$ and $X^T D_\mu (X^T)^* =
D_\mu$, where the transpose of the last equality is $X^* D_\mu X =
D_\mu$.  Therefore, $C=(X^*)^T$ and $X$ commutes with $\Omega$, so
for all $B \in M_k(\C)$ we have
$$\tau(B)= \tr\Big((X^*)^T D_\mu B^T D_\mu\Big) = \tr(D_\mu B D_\mu X^*)=
\tr(X^* B \Omega)$$ and $\sigma(B)=\tau(B)X =
 \tr(X^* B \Omega)X$.

Now assume the hypotheses of the backward direction and define $\tau
\in M_k(\C)^*$ by $\tau(B) = \tr(X^* B \Omega)$, noting that
$\sigma^2=\sigma$ and $\sigma(B) = \tau(B)X$ for all $B \in
M_k(\C)$. Let $\eta, \eta': M_{2k}(\C) \rightarrow M_{2k}(\C)$ be
the maps
\begin{displaymath}
\eta \left( \begin{array}{cc} A & B \\ C & D
 \end{array} \right) = \left( \begin{array}{cc} \rho(A)I & \tau(B)X \\ \tau^*(C)X^*
& \rho(D)I \end{array} \right), \ \
\eta' \left( \begin{array}{cc} A & B \\ C & D
 \end{array} \right) = \left( \begin{array}{cc} \rho(A)I & \tau(B)I \\ \tau^*(C)I
& \rho(D)I \end{array} \right).
\end{displaymath}
Define $\Upsilon: M_{2k}(\C) \rightarrow M_{2k}(\C)$ by
\begin{displaymath}
\Upsilon \left( \begin{array}{cc} A & B \\ C & D  \end{array} \right)
= \left( \begin{array}{cc} X^* & 0 \\ 0 & I  \end{array} \right)
\left( \begin{array}{cc} A & B \\ C & D  \end{array} \right)
\left( \begin{array}{cc} X & 0 \\ 0 & I  \end{array} \right).
\end{displaymath}
Note that $\Upsilon$ and $\Upsilon^{-1}$ are completely positive,
$\Upsilon \circ \eta = \eta'$, and $\Upsilon^{-1} \circ \eta'=\eta$.  Therefore,
$\eta$ is completely positive if and only if $\eta'$ is completely positive.
Since a complex matrix $(m_{ij}) \in M_{r}(\C)$ ($r \in \mathbb{N}$) is positive if and only if
$(m_{ij} I_n) \in M_r(M_n(\C))=M_{rn}(\C)$ is positive for every $n \in \mathbb{N}$, it
follows that
$\eta'$ is completely positive if and only
if $\eta''$ below is completely positive:
\begin{displaymath}
\eta'' \left( \begin{array}{cc} A & B \\ C & D
 \end{array} \right) = \left( \begin{array}{cc} \rho(A) & \tau(B) \\ \tau^*(C)
& \rho(D) \end{array} \right).
\end{displaymath}
Thus, $\eta$ is completely positive if and only if $\eta''$ is. In
other words, $\sigma$ is a corner from $\phi$ to $\phi$ if and only
if $\tau$ is a corner from $\rho$ to $\rho$.  But for all $B \in
M_{k}(\C)$, we have
$$\tau(B) = \sum_{i,j=1}^k c_{ij} \sqrt{\mu_i \mu_j}(e_i, B e_j)$$
for the unitary matrix $C= (X^*)^T$, so $\tau$ is a corner from $\rho$ to $\rho$
by Lemma \ref{corners}.
\end{proof}

\begin{comment} {\bf Remark 2:} Suppose $\phi: M_n(\C) \rightarrow M_n(\C)$ and $\psi: M_{n'}(\C) \rightarrow M_{n'}(\C)$
are unital $q$-positive maps, and let $\nu \in mA$ be normalized,
unbounded, and of the form $\nu( \sqrt{I - \Lambda(1)} B \sqrt{I -
\Lambda(1)}) = (f,Bf)$. Let $\alpha^d$ and $\beta^d$ be the
$E_0$-semigroups induced by $(\phi, \nu)$ and $(\psi, \nu)$. From
the proof of Proposition 4.6 of \cite{Me}, we know that hyper
maximal flow corners from $\alpha$ and $\beta$ correspond to hyper
maximal $q$-corners from $\phi$ to $\psi$. This tells us that
$(\phi, \nu)$ and $(\psi, \nu)$ induce cocycle conjugate
$E_0$-semigroups $\alpha^d$ and $\beta^d$ if and only if there is a
hyper maximal flow corner from $\phi$ to $\psi$. Furthermore, since
the local unitary flow cocycles of an $E_0$-semigroup are in
one-to-one correspondence with the flow corners from $\alpha$ to
$\alpha$, the classification of hyper maximal $q$-corners from
$\phi$ to $\phi$ is a key step in calculating the gauge group of
$\alpha^d$. With this in mind, w
\end{comment}

We will make use of the following standard result regarding completely positive maps,
providing a proof here for the sake of completeness:

\begin{lem}\label{FAF}
Let $\phi: M_n(\C) \rightarrow M_n(\C)$ be a completely positive map.
If $\phi(E)=0$ for a projection $E$, then $\phi(A)=\phi(FAF)$ for all $A \in M_n(\C)$,
where $F=I-E$.
\end{lem}
\begin{proof} We know from \cite{choi} and \cite{arveson} that $\phi$
can be written $\phi(A)= \sum_{i=1}^p S_i A S_i^*$ for some $p \leq n^2$ and 
$\{S_i\}_{i=1}^p \subset M_n(\C)$.  If $\phi(E)=0$ for a projection $E$, then
$$0= S_i E S_i^*=S_i E E S_i^* = (S_i E)(S_i E)^*$$ for all $i$, so $S_iE=ES_i^*=0$ for all $i$.
Therefore, $\phi(EAE)=\phi(EAF)=\phi(FAE)=0$ for every $A \in M_n(\C)$.  Letting $F=I-E$,
we observe that for every $A \in M_n(\C)$,
$$\phi(A)=\phi(EAE + EAF + FAE + FAF)= \phi(FAF).$$\end{proof}

Let $\phi: M_n(\C) \rightarrow M_n(\C)$ and $\psi: M_{n'}(\C)
\rightarrow M_{n'}(\C)$ be unital rank one $q$-positive maps.
We ask two very important questions:  Is there a $q$-corner from
$\phi$ to $\psi$ ?  If so, can we find all such $q$-corners, and,
even further, determine which $q$-corners are hyper maximal?   The
following two theorems give us a complete answer to both questions when
$\phi$ and $\psi$ are implemented by diagonal states.
This suffices, since for any unital rank one $q$-positive maps $\phi$ and $\psi$,
there are always unitaries $U \in M_n(\C)$ and $V \in M_{n'}(\C)$ such that
$\phi_U$ and $\psi_V$ are implemented by diagonal states, where
Remark 1 tells us exactly how to transform the $q$-corners and hyper maximal
$q$-corners from 
$\phi_U$ to $\psi_V$ into those from $\phi$ to $\psi$.

\begin{thm} \label{biggie}
Let $\{\mu_i\}_{i=1}^k$ and $\{r_i\}_{i=1}^{k'}$ be monotonically
decreasing sequences of strictly positive numbers such that
$\sum_{i=1}^k \mu_k = \sum_{i=1}^{k'} r_i =1$. Define unital
$q$-positive maps $\phi: M_n(\C) \rightarrow M_n(\C)$ and $\psi:
M_{n'}(\C) \rightarrow M_{n'}(\C)$ (where $n \geq k$ and $n' \geq k'$) by
\begin{equation}\label{forms} \phi(A) = 
\Big( \sum_{i=1}^k \mu_i a_{ii}\Big)I_n \ \ \textrm{ and }  \ \ \psi(D)
= \Big(\sum_{i=1}^{k'} r_i d_{ii}\Big)I_{n'}\end{equation} for all $A=(a_{ij}) \in
M_n(\C)$ and $D=(d_{ij}) \in M_{n'}(\C)$.  Let $\Omega \in M_k(\C)$
be the trace density matrix for the faithful state $\rho \in
M_k(\C)^*$ defined by $\rho(A)= \sum_{i=1}^k \mu_i a_{ii}$.

If there is a nonzero $q$-corner from $\phi$ to $\psi$, then $k=k'$
and $\mu_i = r_i$ for all $i=1, \ldots, k$.  In that case, a linear
map $\gamma: M_{n,n'}(\C) \rightarrow M_{n,n'}(\C)$ is a $q$-corner
from $\phi$ to $\psi$ if and only if: for some unitary $X \in
M_{k}(\C)$ that commutes with $\Omega$, some contraction $E \in
M_{n-k, n'-k}(\C)$, and some $\lambda \in \C$ with $|\lambda|^2 \leq
\re(\lambda)$, we have
\begin{displaymath}
\gamma \left(  \begin{array}{cc} B_{k,k} & W_{k, n'-k}
\\ Q_{n-k,k} & Y_{n-k, n'-k}     \end{array} \right)
%= \gamma \left(  \begin{array}{cc} 0_{n-k,k} & 0_{n-k, n'-k}
%\\ B_{k,k} & 0_{k, n'-k}    \end{array} \right)
= \lambda \ \tr(X^*  B_{k,k} \Omega)
\left(  \begin{array}{cc} X & 0_{k,n'-k} \\ 0_{n-k,k} & E%_{n-k,n'-k}
  \end{array} \right)
\end{displaymath}
for all
\begin{displaymath}
\left(  \begin{array}{cc} B_{k,k} & W_{k, n'-k}
\\ Q_{n-k,k} & Y_{n-k, n'-k}     \end{array} \right) \in M_{n,n'}(\C).
\end{displaymath}
\end{thm}
\begin{proof}  Suppose that $\gamma$ is a nonzero $q$-corner
from $\phi$ to $\psi$, so $\vartheta: M_{n+n'}(\C) \rightarrow
M_{n + n'}(\C)$ below is $q$-positive:
\begin{displaymath}
\vartheta \left( \begin{array}{cc} A_{n, n} & B_{n , n'}
\\ C_{n' , n} & D_{n' ,n'} \end{array}   \right)
= \left( \begin{array}{cc} \phi(A_{n , n}) & \gamma(B_{n , n'})
\\ \gamma^*(C_{n' , n}) & \psi(D_{n' , n'}) \end{array}   \right) .
\end{displaymath}
We observe that
\begin{displaymath}
L_\vartheta \left( \begin{array}{cc} A_{n , n} & B_{n , n'}
\\ C_{n' , n} & D_{n' , n'} \end{array}   \right)
= \left( \begin{array}{cc} \phi(A_{n , n}) & \sigma(B_{n , n'})
\\ \sigma^*(C_{n' , n}) & \psi(D_{n' , n'}) \end{array}   \right),
\end{displaymath}
where by Lemma \ref{Lphi}, the map $\sigma : = \lim_{t \rightarrow
\infty} t \gamma(I + t \gamma)^{-1}$ is a corner of norm one from
$\phi$ to $\psi$ satisfying $\sigma^2=\sigma$,
$\range(\sigma)=\range(\gamma)$, and $\gamma \circ \sigma = \sigma
\circ \gamma = \gamma$.  Since $||\sigma||=1$, 
Proposition \ref{kk'} implies  $k=k'$ and $r_i = \mu_i$ for all
$i=1, \ldots, k$.

We observe that $L_\vartheta (E)=0$ for the projection 
$$E = \Big(\sum_{i=k+1}^n e_{ii} + \sum_{i=n+k'+1}^{n+n'} e_{ii}\Big) \in M_{n+n'}(\C).$$  Therefore,
$L_\vartheta(A)=L_\vartheta\Big((I-E)A(I-E)\Big)$ for all $A \in M_{n+n'}(\C)$ 
by Lemma \ref{FAF}.
\begin{comment}
$$L_\vartheta(e_{k+1,k+1}) = L_\vartheta(e_{k+2,k+2}) = \cdots
= L_\vartheta(e_{n,n})=0$$ and $$L_\vartheta(e_{n+k+1,n+k+1})=
L_\vartheta(e_{n+k+2, n+k+2}) =
\cdots = L_\vartheta(e_{n',n'})=0,$$  so
complete positivity of $L_\vartheta$ implies that
\end{comment}
In particular, $\sigma$ satisfies
\begin{displaymath}
\sigma \left(  \begin{array}{cc} 0_{k,k} & W_{k, n'-k}
\\ Q_{n-k,k} & Y_{n-k, n'-k}     \end{array} \right) \equiv 0.
\end{displaymath}
In other words, $\sigma$ depends only on its top left $k \times k$
minor, so for some $\tsigma: M_k(\C) \rightarrow M_k(\C)$
and some maps $\ell_i$ from $M_k(\C)$ into the appropriate
matrix spaces, we have
\begin{displaymath}
\sigma \left(  \begin{array}{cc} B_{k,k} & W_{k, n'-k}
\\ Q_{n-k,k} & Y_{n-k, n'-k}     \end{array} \right)
=
\left(  \begin{array}{cc} \tsigma(B_{k,k}) & \ell_1(B_{k,k})
\\ \ell_2(B_{k,k}) & \ell_3(B_{k,k})
   \end{array} \right).
\end{displaymath}
From the facts
$\sigma^2=\sigma$ and $||\sigma||=1$, it follows that
$\tsigma^2=\tsigma$ and $||\tsigma||=1$.

Let $\tphi: M_{k}(\C) \rightarrow M_k(\C)$
be the map $$\tphi(A) = \rho(A)I_k = (\sum_{i=1}^k \mu_i a_{ii})I_k$$ for all $A=
(a_{ij}) \in M_k(\C)$.  Define $\Theta: M_{2k}(\C) \rightarrow M_{2k}(\C)$ by
\begin{displaymath}
\Theta \left( \begin{array}{cc} A_{k,k} & B_{k,k}
\\ C_{k,k} & D_{k,k}  \end{array} \right)
= \left( \begin{array}{cc} \tphi(A_{k,k}) & \tsigma(B_{k,k})
\\ \tsigma ^*(C_{k,k}) & \tpsi(D_{k,k})  \end{array} \right),
\end{displaymath}
and let
\begin{displaymath}
S = \left( \begin{array}{cccc} I_{k,k} & 0_{k,n-k} & 0_{k,k} &
0_{k, n'-k} \\  0_{k,k} & 0_{k,n-k} & I_{k,k} &
0_{k, n'-k}\end{array} \right)
\in M_{2k,n+n'}(\C).
\end{displaymath}
Note that
$$\Theta(N) = S L_\vartheta(S^*NS)S^*$$ for all $N \in M_{2k}(\C)$,
so $\Theta$ is completely positive.  Therefore, $\tsigma$ is a norm
one corner
from $\tphi$ to $\tphi$.
Since $||\tsigma||=1$ and $\tsigma^2=\tsigma$,
Lemma \ref{corner} implies that
for some unitary $X \in M_k(\C)$ that commutes with $\Omega$, we have
\begin{equation}\label{tsig}
\tsigma(B) = \tr(X^* B \Omega)X \end{equation} for all $B \in M_k(\C)$.
For simplicity of notation in what follows, let $\tau \in M_k(\C)^*$ be the functional
$\tau(B) = \tr(X^* B \Omega).$

We claim that $\ell_1=\ell_3 \equiv 0$.  For this, let
\begin{equation} \label{M}
M = \left(  \begin{array}{cc} B_{k,k} & W_{k, n'-k}
\\ Q_{n-k,k} & Y_{n-k, n'-k}     \end{array} \right) \in M_{n,n'}(\C)
\end{equation}
be arbitrary.  We will suppress the subscripts for $B, Q, W$,
and $Y$ for the remainder of the proof.  From \eqref{tsig} and the fact
that $\sigma^2(M)=\sigma(M)$, we have \begin{equation} \label{elli}
\ell_i(B) = \ell_i(\tsigma(B))=\ell_i(\tau(B)X) = \tau(B)\ell_i(X) \end{equation}
for $i=1,2,3$.  Since $\sigma$ is a contraction, it follows that
\begin{displaymath}
1 \geq \Big| \Big| \sigma \left(  \begin{array}{cc} X & 0
\\ 0 & 0  \end{array} \right)\Big|\Big| =\Big|\Big| \left(  \begin{array}{cc} X & \ell_1(X)
\\ \ell_2(X) & \ell_3(X)  \end{array} \right)\Big|\Big|.
\end{displaymath}
But $X$ is unitary, so the line above implies
that $\ell_1(X) = \ell_2(X)=0$,
hence $\ell_1= \ell_2 \equiv 0$ by \eqref{elli}.   Let
$E=\ell_3(X) \in M_{n-k, n'-k}(\C)$, noting that $||E|| \leq 1$
since $\sigma$ is a contraction.  Therefore, $\sigma$ has the form
\begin{displaymath}
\sigma \left(  \begin{array}{cc} B & W
\\ Q & Y     \end{array} \right)
= \tau(B) \left( \begin{array}{cc} X & 0_{k, n'-k}
\\ 0_{n-k,k} & E
 \end{array} \right).
\end{displaymath}
Since $\gamma = \gamma \circ \sigma$ and
\begin{displaymath} \range(\gamma)=\range(\sigma)= \Big\{c
\left(  \begin{array}{cc} X & 0 \\ 0 & E   \end{array} \right): c \in \C
\Big\},
\end{displaymath} we have
\begin{eqnarray*}
\gamma \left(  \begin{array}{cc} B & W
\\ Q & Y      \end{array} \right)
& = & \gamma\Big(\sigma\left(  \begin{array}{cc} B & W
\\ Q & Y     \end{array} \right)\Big)
=
\gamma \left( \tau(B) \Big(\begin{array}{cc} X & 0 \\ 0 & E
\end{array} \Big) \right)
= \tau(B) \
\gamma \left(  \begin{array}{cc} X & 0
\\ 0 & E     \end{array} \right)
\\ & = & \tau(B)  \Big[ \lambda
\left(  \begin{array}{cc} X & 0 \\ 0 & E   \end{array} \right) \Big]
= \lambda \tau(B)  \left(  \begin{array}{cc} X & 0 \\ 0 & E    \end{array} \right)
\end{eqnarray*}
for some $\lambda \in \C$.  Since $\gamma$ is a nonzero $q$-corner between
unital completely positive maps and is thus necessarily
a contraction with no negative eigenvalues, we have
$\lambda \nleq 0$ and $|\lambda| \leq 1$.

In summary: we have proved that if $\gamma$ is a nonzero
$q$-corner,
then it is of the form
\begin{displaymath} \gamma
\left(  \begin{array}{cc} B & W
\\ Q & Y      \end{array} \right) = \lambda \ \tr(X^* B \Omega)
\left(  \begin{array}{cc} X & 0 \\ 0 & E   \end{array} \right)
\end{displaymath}
for some $\lambda \nleq 0$ with $|\lambda| \leq 1$, where $X$ and $E$ satisfy the conditions stated
in the theorem.
To complete the proof, we show
that such a map $\gamma$ is a $q$-corner if and only if $|\lambda|^2 \leq \re(\lambda)$.

Straightforward computations show that for all $t \geq 0$,
\begin{displaymath}
(I + t \gamma)^{-1}\left(  \begin{array}{cc} B & W
\\ Q & Y     \end{array} \right)=
\left(  \begin{array}{cc}
B - \frac{t \lambda \tau(B)}{1+t\lambda} X & W \\
Q & Y - \frac{t\lambda \tau(B)}{1+t \lambda} E  \end{array} \right)
\end{displaymath}
and
\begin{displaymath}
\gamma(I + t \gamma)^{-1}\left(  \begin{array}{cc} B & W
\\ Q & Y     \end{array} \right)=
\left(  \begin{array}{cc} \frac{\lambda \tau(B)}{1+t\lambda} X & 0 \\
0 & \frac{\lambda \tau(B)}{1+t \lambda} E   \end{array} \right)
= \frac{\lambda}{1+t \lambda} \gamma \left(  \begin{array}{cc} B & W
\\ Q & Y     \end{array} \right).
\end{displaymath}

For each $t \geq 0$, define maps $\Theta_t: M_{2k}(\C) \rightarrow M_{2k}(\C),$
$L_t: M_{2k}(\C) \rightarrow M_{n+n'-2k}(\C)$,
and $\Upsilon_t: M_{2k}(\C) \rightarrow M_{n+n'-2k}(\C)$ by

\begin{displaymath} \Theta_t \left( \begin{array}{cc} A & B
\\ C & D \end{array} \right) =
\left( \begin{array}{cc}
\frac{1}{1+t} \rho(A)I_{k,k} & \frac{\lambda}{1+t \lambda} \tau(B) X
\\ \frac{\bar{\lambda}}{1+t \bar{\lambda}} \tau^*(C) X^* & \frac{1}{1+t} \rho(D)I_{k,k}
 \end{array} \right),
\end{displaymath}

\begin{displaymath}
L_t \left( \begin{array}{cc} A & B
\\ C & D \end{array} \right) = \left( \begin{array}{cc}
\frac{1}{1+t} \rho(A)EE^* & \frac{\lambda}{1+t \lambda} \tau(B) E
\\ \frac{\bar{\lambda}}{1+t \bar{\lambda}} \tau^*(C) E^* & \frac{1}{1+t} \rho(D)I_{n'-k, n'-k}
 \end{array} \right),
\end{displaymath}
and
\begin{displaymath}
\Upsilon_t \left( \begin{array}{cc} A & B
\\ C & D \end{array} \right) =
L_t \left( \begin{array}{cc} A & B
\\ C & D \end{array} \right) + \left( \begin{array}{cc}
\frac{1}{1+t} \rho(A)(I_{n-k,n-k}-EE^*) & 0_{n-k,n'-k} \\ 0_{n'-k,n-k} & 0_{n'-k,n'-k}
 \end{array} \right).
\end{displaymath}
Let
\begin{displaymath}
T= \left( \begin{array}{cccc} 0_{n-k,k} & I_{n-k, n-k} & 0_{n-k,k} & 0_{n-k,n'-k}
\\ 0_{n'-k,k} & 0_{n'-k,n-k} & 0_{n'-k,k} & I_{n'-k,n'-k}  \end{array} \right)
\in M_{n+n'-2k, n+n'}(\C),
\end{displaymath}
and let $M \in M_{n+n'}(\C)$
be arbitrary, writing
\begin{displaymath}
M = \left( \begin{array}{cccc} A_{k,k} & q_{k,n-k} & B_{k,k} & r_{k,n'-k}
\\ s_{n-k,k} & t_{n-k,n-k} & u_{n-k,k} & v_{n-k,n'-k}
\\ C_{k,k} & w_{k,n-k} & D_{k,k} & c_{k, n'-k}
\\ d_{n'-k,k} & e_{n'-k,n-k} & f_{n'-k,k} & g_{n'-k, n'-k}, \end{array} \right), \
\textrm{ so } \ SMS^* = \left( \begin{array}{cccc} A_{k,k} & B_{k,k} \\ C_{k,k} & D_{k,k} 
\end{array} \right).
\end{displaymath}
For every $t \geq 0$, we have
\begin{eqnarray} \nonumber
\vartheta(I + t \vartheta)^{-1}(M)
& = & \left( \begin{array}{cccc}
\frac{1}{1+t} \rho(A)I_{k,k} & 0_{k,n-k} & \frac{\lambda}{1+t \lambda} \tau(B) X & 0_{k,n-k}
\\ 0_{n-k,k} & \frac{1}{1+t} \rho(A)I_{n-k,n-k} & 0_{n-k,k} & \frac{\lambda}{1+t \lambda} \tau(B) E
\\ \frac{\bar{\lambda}}{1+t \bar{\lambda}} \tau^*(C) X^* & 0_{k,n-k} & \frac{1}{1+t}\rho(D)I_{k,k} & 0_{k, n'-k}
\\ 0_{n'-k,k} & \frac{\bar{\lambda}}{1+t \bar{\lambda}} \tau^*(C)E^*
& 0_{n'-k,k} & \frac{1}{1+t}\rho(D)I_{n'-k, n'-k} \end{array} \right)
\\ \label{oneq} & = & S^* \Theta_t(SMS^*)S + T^* \Upsilon_t(SMS^*) T.
 \end{eqnarray}
Note also that for all $N \in M_{2k}(\C)$,
\begin{equation}
\label{twoeq} \Theta_t(N) = S \Big(\vartheta(I + t \vartheta)^{-1}(S^*NS)\Big) S^*
, \ \ \Upsilon_t(N) = T \Big(\vartheta(I + t \vartheta)^{-1}(S^*NS)\Big) T^*.
\end{equation}
It follows from \eqref{oneq} and \eqref{twoeq} that $\vartheta$ is $q$-positive
if and only if
$\Theta_t$ and $\Upsilon_t$ are completely positive for all $t \geq 0$.

We may easily argue as in the proof of Lemma \ref{corner} to conclude that
$\Theta_t$ is completely positive for all $t \geq 0$ if and only if the maps $\eta_t'':
M_{2k}(\C) \rightarrow M_{2}(\C)$ below are
completely positive for all $t \geq 0$:
\begin{displaymath}
\eta_t '' \left( \begin{array}{cc} A & B
\\ C & D \end{array} \right) = \left( \begin{array}{cc}
\frac{1}{1+t} \rho(A) & \frac{\lambda}{1+t \lambda} \tau(B)
\\ \frac{\bar{\lambda}}{1+t \bar{\lambda}} \tau^*(C) & \frac{1}{1+t} \rho(D)
 \end{array} \right).
\end{displaymath}
Recall that in the proof of Lemma \ref{corner}, we showed
that $\tau$ is a corner from $\rho$ to $\rho$.  Since $||\rho||=||\tau||=1$, it follows from 
Lemma \ref{corners}
that $c \tau$ is a corner from $\rho$ to $\rho$ if and only if $|c| \leq 1$.  Since
\begin{displaymath}
(1+t) \eta_t '' \left( \begin{array}{cc}   A & B \\
C & D
\end{array} \right)
=
\left( \begin{array}{cc}   \rho(A) &
\frac{\lambda(1+t)}{1+t \lambda} \tau(B) \\ \frac{\bar{\lambda}(1+t)}
{1+ \bar{\lambda}} \tau^*(C) & \rho(D)
\end{array} \right),
\end{displaymath}
we see that $\eta_t ''$ is completely positive for all $t \geq 0$ if and only if
$$\Big|\frac{\lambda(1+t)}{1+t \lambda}\Big| \leq 1 \ \ \ (\textrm{where we already know } \lambda \nleq 0 
\textrm{ and } |\lambda| \leq 1)$$
for all $t \geq 0$.
Squaring both sides of the above equation and then cross multiplying gives us
$$|\lambda|^2(1+2t+t^2) \leq 1 + 2t \re(\lambda) + |\lambda|^2 t^2, \ \ (\lambda \nleq 0, \
|\lambda| \leq 1)$$
which is equivalent to
\begin{equation}\label{lambdat} |\lambda|^2 \leq \frac{1 + 2t \re(\lambda)}{1 +
2t} \ \ (\lambda \nleq 0, \
|\lambda| \leq 1) \end{equation} for all nonnegative $t$.  Note that if $|\lambda|^2 \leq \re(\lambda)$,
then $\re(\lambda) \leq 1$ and equation \eqref{lambdat} holds for $t \geq 0$.  On the other hand,
suppose that $\lambda$ is any complex number that satisfies \eqref{lambdat} for all $t \geq 0$.
We conclude immediately that $\re(\lambda) > 0$, whereby the fact that
$|\lambda| \leq 1$ implies $\re(\lambda) \in
(0,1]$.  A computation shows that the net
$\{\frac{1 + 2t \re(\lambda)}{1 + 2t}\}_{t \geq 0}$ is monotonically
decreasing and converges to $\re(\lambda)$, hence $|\lambda|^2 \leq \re(\lambda)$
by \eqref{lambdat}.  We have now shown that $\eta_t ''$ (and thus $\Theta_t$) is completely positive for all $t \geq 0$
if and only if $|\lambda|^2 \leq \re(\lambda)$.  Therefore, if $|\lambda|^2 > \re(\lambda)$
then \eqref{twoeq} implies that
$\vartheta$ is not $q$-positive, which is to say that $\gamma$ is not a $q$-corner
from $\phi$ to $\psi$.

Suppose that $|\lambda|^2 \leq \re(\lambda)$.  Then from above, the maps $\{\Theta_t\}_{t \geq 0}$
are all completely positive.  Let
\begin{displaymath} G=
\left( \begin{array}{cc}
E & 0_{n-k, n'-k}
\\ 0_{n'-k, n'-k} & I_{n'-k}
\end{array} \right) \in M_{n+n'-2k, 2n' - 2k}(\C).
 \end{displaymath}
We observe that
\begin{eqnarray*}
(1+t) L_t \left( \begin{array}{cc} A & B
\\ C & D \end{array} \right)
 & = & G \left(
\begin{array}{cc} \rho(A)I_{n'-k} &
\frac{\lambda(1+t)}
{1+t \lambda} \tau(B) I_{n'-k} \\ \frac{\bar{\lambda}
(1+t)}
{1 + t \bar{\lambda}} \tau^*(C) I_{n'-k} &
\rho(D)I_{n'-k} \end{array} \right)
G^*,
\end{eqnarray*}
where we have already shown that the map in the middle is completely
positive since $|\lambda|^2 \leq \re(\lambda)$.  Thus, $L_t$ is completely
positive for every $t \geq 0$.  Also, $\Upsilon_t - L_t$
has the form
\begin{displaymath}
(\Upsilon_t - L_t)
\left( \begin{array}{cc} A & B
\\ C & D \end{array} \right) =
\left( \begin{array}{cc} \rho(A)(I_{n-k}-EE^*) & 0_{n-k,n'-k}
\\ 0_{n'-k,n-k} & 0_{n'-k,n'-k} \end{array} \right),
\end{displaymath}
where the right hand side is completely positive since $||E|| \leq 1$.  Therefore,
the maps $\{\Upsilon_t \}_{t \geq 0}$ are all completely positive,
so \eqref{oneq} implies that $\vartheta(I + t \vartheta)^{-1}$ is completely positive
for all $t \geq 0$, hence
$\gamma$ is a $q$-corner from $\phi$ to $\psi$.
\end{proof}

\begin{thm} \label{bigone}
Assume the notation of the previous theorem, and suppose that $k=k'$ and $\mu_i=r_i$
for all $i=1, \ldots, k$.  A $q$-corner $\gamma: M_{n,n'}(\C)
\rightarrow M_{n,n'}(\C)$ from $\phi$ to $\psi$ is hyper maximal if and only if
$n=n'$, $0<|\lambda|^2 = \re(\lambda)$, and $E$
is unitary.
\end{thm}
\begin{proof}  We first show that $\gamma$ is not hyper maximal if $n \neq n'$,
regardless of the assumptions for $\lambda$ or $E$.
If $n>n'$, then $EE^* \in M_{n-k}(\C)$ is a positive contraction of rank
at most $n'-k$, so $EE^* \neq I_{n-k}$.

Define $\phi': M_{n}(\C) \rightarrow M_{n}(\C)$
by
\begin{displaymath} \phi'(R)=\phi(R) \left( \begin{array}{cc}
I_{k,k} & 0_{k, n-k} \\ 0_{n-k,k} & EE^*
\end{array} \right), \end{displaymath}
observing that $\phi'(I + t \phi')^{-1} = (1/(1+t)) \phi'$ for all $\geq 0$.
Define $\vartheta': M_{n+n'}(\C)
\rightarrow M_{n+n'}(\C)$ by
\begin{displaymath}
\vartheta'
= \left( \begin{array}{cc} \phi' & \gamma
\\ \gamma^* & \psi  \end{array} \right),
\end{displaymath}
noting that $\vartheta'$ has no negative eigenvalues.
Writing each $M \in M_{n+n'}(\C)$ in the form \eqref{M}, we see
\begin{eqnarray} \nonumber
\vartheta'(I + t \vartheta')^{-1}(M)
& = & \left( \begin{array}{cccc}
\frac{1}{1+t} \rho(A)I_{k,k} & 0_{k,n-k} & \frac{\lambda}{1+t \lambda} \tau(B) X & 0_{k,n-k}
\\ 0_{n-k,k} & \frac{1}{1+t} \rho(A)EE^* & 0_{n-k,k} & \frac{\lambda}{1+t \lambda} \tau(B) E
\\ \frac{\bar{\lambda}}{1+t \bar{\lambda}} \tau^*(C) X^* & 0_{k,n-k} & \frac{t}{1+t}\rho(D)I_{k,k} & 0_{k, n'-k}
\\ 0_{n'-k,k} & \frac{\bar{\lambda}}{1+t \bar{\lambda}} \tau^*(C)E^*
& 0_{n'-k,k} & \frac{1}{1+t}\rho(D)I_{n'-k, n'-k} \end{array} \right)
\\ & = &  \label{vee'} S^* \Theta_t(SMS^*)S + T^* L_t(SMS^*) T.
 \end{eqnarray}
for every $t \geq 0$, hence $\vartheta'$ is $q$-positive.
By \eqref{oneq} and \eqref{vee'}, we have
$$\Big(\vartheta(I + t \vartheta)^{-1}
- \vartheta'(I + t \vartheta')^{-1}\Big)(M) = T \Big(\Big(\Upsilon_t-L_t
\Big)(S^*MS)\Big) T^*.$$
Since $\Upsilon_t - L_t$ is completely positive
for all $t\geq 0$ (as shown in the previous proof), the above equation implies
that $\vartheta \geq_q \vartheta'$.  However, $\vartheta \neq \vartheta'$
since $EE^* \lneq I_{n-k}$, so $\gamma$ is not hyper maximal.

If $n< n'$, then since $E^*E \lneq I_{n'-k}$, we may replace
$\{L_t\}_{t=0}^\infty$ with the maps $\{R_t\}_{t=0}^\infty$ below
and argue analogously (this time cutting down $\psi$ using $E^*E$)
to show that $\gamma$ is not hyper maximal:
\begin{displaymath}
R_t \left( \begin{array}{cc} A_{k,k} & B_{k,k}
\\ C_{k,k} & D_{k,k}  \end{array} \right) =
\left( \begin{array}{cc} \frac{1}{1+t}\rho(A) I_{n-k} &  \frac{\lambda}
{1+t \lambda} \tau(B) E
\\ \frac{\bar{\lambda}}{1 + t \bar{\lambda}} \tau^*(C) E^* & \frac{1}{1+t}\rho(D)
E^*E \end{array} \right).
\end{displaymath}
Of course, if $n=n'$ but $E$ is not unitary, then $EE^* \lneq I_{n-k}$, and the same
argument given in the case that $n>n'$ shows that $\gamma$ is not hyper maximal.

Therefore, we may suppose for the remainder of the proof that $n=n'$ and
$E$ is unitary.  Note that $\phi = \psi$ since $n=n'$.  For some $a \in (0,1]$, we have $|\lambda|^2= a \re(\lambda)$.
We first show that $\gamma$ is not hyper maximal if $a \neq 1$.
We claim that the map
$\vartheta'': M_{2n}(\C) \rightarrow M_{2n}(\C)$ defined by
\begin{displaymath}
\vartheta'' \left( \begin{array}{cc} A_{n,n} & B_{n,n}
\\ C_{n,n} & D_{n,n}  \end{array} \right)
= \left( \begin{array}{cc} a \phi(A_{n,n}) & \gamma(B_{n,n})
\\ \gamma^*(C_{n,n}) & a \phi(D_{n,n})  \end{array} \right)
\end{displaymath}
satisfies $\vartheta'' \geq_q 0$.  For each $t \geq 0$, let $\eta_t^{(a)}:
M_{2k}(\C) \rightarrow M_{2}(\C)$ be the map
\begin{displaymath}
\eta_t^{(a)} \left( \begin{array}{cc} A & B
\\ C & D \end{array} \right) =
\left( \begin{array}{cc}
\frac{a}{1+at} \rho(A) & \frac{\lambda}{1+t \lambda} \tau(B)
\\ \frac{\bar{\lambda}}{1+t \bar{\lambda}} \tau^*(C) & \frac{a}{1+at} \rho(D)
\end{array} \right).
\end{displaymath}
It is routine to check
that since $\tau$ is a corner from $\rho$ to $\rho$,
the condition $|\lambda|^2 = a \re(\lambda)$ implies that
$\frac{\lambda}{1+t \lambda} \tau$ is a corner from $\frac{a}{1+at} \rho$ to
$\frac{a}{1+at} \rho$
for every $t \geq 0$, so $\eta_t^{(a)}$ is completely positive for all $t \geq 0$.
Defining $\Theta_t ^{(a)}$ and $\Upsilon_t^{(a)}$ for each $t \geq 0$ by
\begin{displaymath} \Theta_t^{(a)} \left( \begin{array}{cc} A & B
\\ C & D \end{array} \right) =
\left( \begin{array}{cc}
\frac{a}{1+at} \rho(A)I_{k} & \frac{\lambda}{1+t \lambda} \tau(B) X
\\ \frac{\bar{\lambda}}{1+t \bar{\lambda}} \tau^*(C) X^* & \frac{a}{1+at} \rho(D)I_{k}
 \end{array} \right)
\end{displaymath}
and
\begin{eqnarray*}
\Upsilon_t^{(a)} \left( \begin{array}{cc} A & B
\\ C & D \end{array} \right)
 & = & G \left(
\begin{array}{cc} \frac{a}{1+at} \rho(A)I_{n-k} &
\frac{\lambda}
{1+t \lambda} \tau(B) I_{n-k} \\ \frac{\bar{\lambda}
}
{1 + t \bar{\lambda}} \tau^*(C) I_{n-k} &
\frac{a}{1+at} \rho(D)I_{n-k} \end{array} \right)
G^* 
\\ & = & \left( \begin{array}{cc}
\frac{a}{1+at} \rho(A)I_{n-k} & \frac{\lambda}{1+t \lambda} \tau(B) E
\\ \frac{\bar{\lambda}}{1+t \bar{\lambda}} \tau^*(C) E^* & \frac{a}{1+at} \rho(D)I_{n-k}
 \end{array} \right) ,
\end{eqnarray*}
we observe that the maps $\{\Theta_t^{(a)}\}_{t \geq 0}$ and $\{\Upsilon_t^{(a)}\}_{t \geq 0}$
are all completely positive since $\eta_t^{(a)}$ is completely positive
for all $t \geq 0$.  Note that $$(a\phi)(I + t a\phi)^{-1} =
\frac{a}{1+at} \phi$$ for all $t \geq 0$, so
for every $M \in M_{2n}(\C)$, we have
$$\vartheta''(I + t \vartheta'')^{-1}(M) =
S \Big(\Theta_t^{(a)}(S^*MS) \Big) S^*
+ T^* \Big(\Upsilon_t ^{(a)}(S^*MS) \Big) T.$$
Therefore, $\vartheta'' \geq_q 0$, and trivially $\vartheta \geq_q \vartheta''$.
If $a \neq 1$, then $\vartheta'' \neq \vartheta$, hence $\gamma$ is not
hyper maximal.  To finish the proof, it suffices to show
that $\gamma$ is hyper maximal if $a=1$ (of course, maintaining
our assumption that $E$ is unitary).

Suppose $a=1$, and let $\phi'$ be any $q$-subordinate of $\phi$
such that
\begin{displaymath}
\chi := \left( \begin{array}{cc}  \phi' & \gamma
\\ \gamma^* & \phi
\end{array} \right) \geq_q 0.
\end{displaymath}
If $L_{\phi'}(I) \neq I$, then $L_{\phi'}(I)= R \lneq I$ for some positive
$R \in M_n(\C)$.  Letting $Z$ be the unitary matrix
\begin{displaymath}
Z= \left( \begin{array}{cc} X & 0_{k,n-k}
\\ 0_{n-k,k} & E \end{array} \right) \in M_{n-k}(\C),
\end{displaymath} we observe that
\begin{equation}\label{duh}
0 \leq L_\chi \left( \begin{array}{cc}  I & Z
\\ Z^* & I
\end{array} \right) = \left( \begin{array}{cc}  R & Z
\\ Z^* & I
\end{array} \right).
\end{equation}
Since $R \lneq I$, we have $(f,Rf) < 1$ for some unit vector $f \in \C^n$.
A quick calculation shows that
\begin{displaymath}
\Big\langle  \left( \begin{array}{c} f \\ -Z^* f \end{array} \right)
, \left( \begin{array}{cc}  R & Z
\\ Z^* & I
\end{array} \right) \left( \begin{array}{c} f \\ -Z^* f \end{array} \right)
\Big\rangle = (f, Rf) -1 <0,
\end{displaymath}
contradicting \eqref{duh}.

Therefore, $L_{\phi'}(I)=I$.  Since $\phi \geq_q \phi'$, it follows that $L_\phi -
L_{\phi'}$ is completely positive,  so $$||L_\phi - L_{\phi'}|| =||L_\phi(I) - L_{\phi'}(I)|| = 0,$$ hence
$L_{\phi'}(A)= L_\phi(A)=\phi(A) = \ell(A) I$ for the state $\ell \in M_{n}(\C)^*$ defined by
$\ell(A) = \sum_{i=1}^k \mu_i a_{kk}$.  But $\range(\phi')= \range(L_{\phi'})=\{cI: c \in \C\}$
and $\phi' = \phi' \circ L_{\phi'}$, so $\phi'(I)=rI$ for some $r \leq 1$ and
$$\phi'(A) = \phi'(L_{\phi'}(A)) = \phi( \ell(A)I)=\ell(A) \phi'(I) = r \ell(A) I= r \phi(A)$$
for all $A \in M_n(\C)$.

We claim that $r=1$.  To prove this, we define $V_t: M_{2k}(\C)
\rightarrow M_{2k}(\C)$ for each $t \geq 0$ by

\begin{displaymath}
V_t \left( \begin{array}{cc} A & B \\ C & D \end{array} \right)
= S \left( \chi(I + t \chi)^{-1} \Big[
S^* \Big( \begin{array}{cc} A & B \\ C & D \end{array} \Big)  S \Big]\right)
 S^* =
 \left( \begin{array}{cc} \frac{r}{1+rt} \rho(A)I_k & \frac{\lambda}{1+t \lambda} \tau(B)X
\\ \frac{\bar{\lambda}}{1+t \bar{\lambda}}\tau^*(C)X^* & \frac{1}{1+t} \rho(D)I_k \end{array} \right).
\end{displaymath}
Since $\chi \geq_q 0$, each $V_t$ is completely positive.  Therefore,
\begin{displaymath}
0 \leq \left( \begin{array}{cc} X^* & 0 \\ 0 & I  \end{array} \right)
\Big[ V_t \left( \begin{array}{cc} I & X \\ X^* & I \end{array} \right)
\Big]
\left( \begin{array}{cc} X & 0 \\ 0 & I  \end{array} \right)
= \left( \begin{array}{cc} \frac{r}{1+rt} I & \frac{\lambda}{1+t \lambda}
I  \\ \frac{\bar{\lambda}}{1+t \bar{\lambda}}I & \frac{1}{1+t} I  \end{array} \right),
\end{displaymath}
hence
$$\frac{r}{(1+rt)(1+t)} \geq \frac{|\lambda|^2}{|1+t \lambda|^2}
= \frac{\re(\lambda)}{1+(t^2+2t)\re(\lambda)}$$
for all $t \geq 0$.  This is equivalent to
\begin{equation}
\label{r} r \geq \frac{(1+t)\re(\lambda)}{1+t\re(\lambda)}\end{equation} for all $t \geq 0$.
We take the limit as $t \rightarrow \infty$
in \eqref{r} and observe $r \geq 1$.  Since $r \leq 1$ we have $r=1$, so $\phi'=\phi$.

We have shown that if
\begin{displaymath}
 \left( \begin{array}{cc}  \phi & \gamma
\\ \gamma^* & \phi
\end{array} \right) \geq_q  \left( \begin{array}{cc}  \phi' & \gamma
\\ \gamma^* & \phi
\end{array} \right) \geq_q 0,
\end{displaymath}
then $\phi=\phi'$.  An analogous argument shows that
if
\begin{displaymath}
 \left( \begin{array}{cc}  \phi & \gamma
\\ \gamma^* & \phi
\end{array} \right) \geq_q  \left( \begin{array}{cc}  \phi & \gamma
\\ \gamma^* & \phi'
\end{array} \right) \geq_q 0,
\end{displaymath} then $\phi=\phi'$.  Therefore, $\gamma$ is hyper maximal.
\end{proof}

We are now ready to prove the following:
\begin{thm} \label{theone}
Let $\phi: M_n(\C) \rightarrow M_n(\C)$ and $\psi: M_{n'}(\C)
\rightarrow M_{n'}(\C)$ be rank one unital $q$-positive maps, and
let $\nu$ be a type II Powers weight of the form $$\nu(\sqrt{I -
\Lambda(1)}B \sqrt{I - \Lambda(1)}) = (f,Bf).$$
The $E_0$-semigroups induced by $(\phi, \nu)$ and $(\psi, \nu)$
are cocycle conjugate if and only if $n=n'$ and $\phi$ is conjugate to $\psi$.
\end{thm}
\begin{proof} The backward direction follows trivially from Proposition
\ref{arrgh}.  For the forward direction, suppose $(\phi, \nu)$ and $(\psi, \nu)$
induce cocycle conjugate $E_0$-semigroups $\alpha^d$ and $\beta^d$.  
For some sets $\{\mu_i\}_{i=1}^k$ and $\{r_i\}_{i=1}^{k'}$ satisfying the conditions
of Theorem \ref{biggie}
and some unitaries $U \in M_n(\C)$ and $V \in M_{n'}(\C)$, $\phi_U$
and $\psi_V$ have the form of \eqref{forms}.  Let $\alpha_U ^d$ and $\beta_V ^d$
be the $E_0$-semigroups induced by $(\phi_U, \nu)$ and $(\psi_V, \nu)$, respectively.
Since $\alpha_U ^d \simeq \alpha^d$
and $\beta_V ^d \simeq \beta^d \simeq \alpha^d$, we have $\alpha_U ^d \simeq \beta_V ^d$,
so by Proposition \ref{hypqc}, there is a hyper maximal $q$-corner from $\phi_U$ to $\psi_V$.
Theorems \ref{biggie} and \ref{bigone} imply that $n=n'$, $k=k'$, and 
$\mu_i=r_i$ for all $i=1, \ldots, k$.  In other words, 
$\phi_U=\psi_V$.  Therefore, $\phi=\psi_{(VU^*)}$, so $\phi$ and $\psi$ are conjugate.
\end{proof}

\end{document}